\documentclass[12pt,a4paper,reqno]{amsart}
\usepackage[utf8]{inputenc}
\usepackage{mathptmx}

\usepackage[english]{babel}

\usepackage{amsmath, amsthm, amsfonts,amssymb}
\usepackage{hyperref}

\usepackage{xcolor}
\newtheorem{thm}{Theorem}[section]
\newtheorem{cor}[thm]{Corollary}

\newtheorem{prop}[thm]{Proposition}

\theoremstyle{plain}

\theoremstyle{definition}

\newtheorem{defn}[thm]{Definition}
\theoremstyle{remark}

\newcommand\blfootnote[1]{%
	\begingroup
	\renewcommand\thefootnote{}\footnote{#1}%
	\addtocounter{footnote}{-1}%
	\endgroup
}

\usepackage{vmargin}

\setpapersize{A4}
\setmargins{2.5cm}       
{1.5cm}                        
{16.5cm}                      
{23.42cm}                    
{10pt}                           
{1cm}                           
{0pt}                             
{2cm}                           

\newcommand {\Id} {\operatorname{Id}}

\newcommand{\N}{\mathbb{N}}
\newcommand{\R}{\mathbb{R}}

\title{A remark on approximation with polynomials and greedy bases}
\date{}
\begin{document}
	\author{Pablo M. Bern\'a, Antonio P\'erez}
	\address{Pablo M. Bern\'a
		\\
		Departmento de Matem\'aticas
		\\
		Universidad Aut\'onoma de Madrid
		\\
		28049 Madrid, Spain} \email{pablo.berna@uam.es}
	
	\address{Antonio P\'erez
		\\
		Instituto de Ciencias Matem\'aticas (CSIC-UAM-UC3M-UCM),
		\\
		C/ Nicol\'as Cabrera
13-15, Campus de Cantoblanco
		\\
		28049 Madrid, Spain} \email{antonio.perez@icmat.es}
	\maketitle

	\begin{abstract}
We investigate properties of the $m$-th error of approximation by polynomials with constant coefficients $\mathcal{D}_{m}(x)$ and with modulus-constant coefficients $\mathcal{D}_{m}^{\ast}(x)$ introduced by Bern\'a and Blasco (2016) to study greedy bases in Banach spaces. We characterize when $\liminf_{m}{\mathcal{D}_{m}(x)}$ and $\liminf_{m}{\mathcal{D}_{m}^*(x)}$ are equivalent to $\| x\|$ in terms of the democracy and superdemocracy functions, and provide sufficient conditions ensuring that $\lim_{m}{\mathcal{D}_{m}^*(x)} = \lim_{m}{\mathcal{D}_{m}(x)} = \| x\|$, extending previous very particular results.
\end{abstract}

	\blfootnote{\hspace{-0.031\textwidth} 2000 Mathematics Subject Classification. 46B15, 41A65.\newline
	The first author was supported by a PhD fellowship of the program ``Ayudas para contratos predoctorales para la formación de doctores 2017" (MINECO, Spain) and the grants MTM-2016-76566-P (MINECO, Spain) and 19368/PI/14 (\emph{Fundaci\'on S\'eneca}, Regi\'on de Murcia, Spain). Also, the first author would like to thank the Isaac Newton Institute for Mathematical Sciences, Cambridge, for  hospitality during the program Approximation, Sampling and Compression in Data Science where some work on this paper was undertaken. The second author is acknowledges financial support from the Spanish Ministry of
Economy and Competitiveness, through the ``Severo Ochoa Programme for Centres of Excellence in R\&D'' (SEV-2015-0554).}
	
	\begin{section}{Introduction}\label{section1}
		Let $(\mathbb X,\Vert \cdot \Vert)$ be a real Banach space and let  $\mathcal{B}=(e_n)_{n=1}^{\infty}$ be a semi-normalized  (Schauder) basis of $\mathbb X$ with biorthogonal functionals $(e_n^{*})_{n=1}^{\infty}$, that is:
		\begin{enumerate}\itemsep0.3em 
		\item[(i)] There exist $a,b > 0$ such that $a  \leq \| e_{n}\|, \|e_{n}^{\ast}\| \leq b$ for every $n \in \N$,
		\item[(ii)] $e_{k}^{\ast}(e_{n}) = \delta_{k n}$ for every $k,n \in \N$,
		\item[(iii)] The sequence of projections 
	$P_{m}: \mathbb{X} \longrightarrow \mathbb{X}$ given by
\[ P_{m}(x) = \sum_{n=1}^{m}{e_{n}^{\ast}(x) \, e_{n}}\,\,, \quad x \in \mathbb{X}\,\]
satisfy $\lim_{n}{\|P_{m}(x) - x\|} = 0$ for every $x \in \mathbb{X}$. In this case, the \emph{basis constant} of $\mathcal{B}$ is
\[ K_{b}:=\sup_{m \in \N} \| P_{m}\| < \infty\,. \] 
We say that $\mathcal{B}$ is \emph{monotone} whether $K_{b} = 1$.
		\end{enumerate}
		Along the paper we will refer to every such $\mathcal{B}$ simply as a \emph{basis}. Of course, as $m$ increases $P_{m}(x)$ offers a good approximation of $x$ by linear combinations of $m$-elements of the basis, but it is natural to ask whether a suitable (and systematic) rearrangement can provide better convergence rates. A natural proposal is the  \textit{Thresholding Greedy Algorithm} (TGA) introduced by S. V. Konyagin and V. N. Temlyakov (\cite{KT}): given $x \in \mathbb{X}$ we first consider the rearranging function $\rho: \N \longrightarrow \N$ satisfying that if $j<k$ then either $\vert e_{\rho(j)}^*(x)\vert > \vert e_{\rho(k)}^*(x)\vert$ or $\vert e_{\rho(j)}^*(x)\vert = \vert e_{\rho(k)}^*(x)\vert$ and $\rho(j)<\rho(k)$. The $m$-\textit{th greedy sum} of $x$ is then $$\mathcal{G}_m(x) \; = \; \sum_{j=1}^m e_{\rho(j)}^*(x) \, e_{\rho(j)}  \; = \; \sum_{k \in \Lambda_m(x)} e_k^*(x) e_k\,,$$
where $\Lambda_m(x) = \{ \rho(n) : n \leq m\}$
	is the \textit{greedy set} of $x$ with cardinality $m$. Related to this, S. V. Konyagin and V. N. Temlyakov defined in \cite{KT}  the concepts of \textit{greedy} and \textit{quasi-greedy} bases.
	 \begin{defn}
	 We say that $\mathcal B$ is \emph{quasi-greedy} if there exists a positive constant $C_q$ such that $$\Vert x-\mathcal G_m(x)\Vert \leq C_q\Vert x\Vert,\; \forall x\in\mathbb X, \forall m\in\mathbb N.$$
	 \end{defn}
P. Wojtaszczyk proved in \cite{Woj} that quasi-greediness is equivalent to the convergence of the algorithm, that is, $\mathcal B$ is quasi-greedy if and only if 
 $$\lim_{m\rightarrow+\infty}\Vert x-\mathcal G_m(x)\Vert = 0,\; \forall x\in\mathbb X.$$
   
	 \begin{defn}\label{defn:greedy}
	 We say that $\mathcal B$ is \emph{greedy} if there exists a positive constant $C$ such that
	 	\begin{equation}\label{equa:defnGreedyAux1}
	 	\Vert x-\mathcal G_m(x)\Vert \leq C\sigma_m(x),\; \forall x\in\mathbb X, \forall m\in\mathbb N,
	 	\end{equation}
	 	where
	 	$$\sigma_m(x,\mathcal B)_{\mathbb X}=\sigma_m(x):=\inf\left\lbrace \left\Vert x-\sum_{n\in A}a_n e_n\right\Vert : a_n\in\mathbb F, A\subset\mathbb N, \vert A\vert= m\right\rbrace.$$
	 \end{defn}
	Konyagin and Temlykov \cite{KT} showed that, although every greedy basis is quasigreedy, the converse does not holds (see also \cite[Section 10.2]{AK}). They also characterize greedy bases as those which are unconditional and democratic. To define the last notion we have to introduce some notation. For each finite subset $A \subset \N$ and every scalar sequence $\varepsilon = (\varepsilon_{n})$ with $|\varepsilon_{n}| = 1$ for each $n \in \N$ (from now on we will write $|\varepsilon| = 1$, for simplicity) let us denote
	\[ \mathbf{1}_{A} := \sum_{n \in A}{e_{n}} \quad \quad \mbox{ and } \quad \quad \mathbf{1}_{\varepsilon A}:= \sum_{n \in A}{\varepsilon_{n} \, e_{n}}\,. \]
As usual,  $|A|$ stands for the cardinal of  $A$. We then define the \emph{democracy functions}	as
\[ h_{l}(m) = \inf_{|A| = m, |\varepsilon| = 1}{\| \mathbf{1}_{\varepsilon A}\|} \quad \mbox{, } \quad   h_{r}(m) = \sup_{|A| = m, |\varepsilon| = 1}{\| \mathbf{1}_{\varepsilon A}\|}\, \quad\quad  (m \in \N)\,. \]
and the \emph{superdemocracy functions} as
\[ h_{l}^*(m) = \inf_{|A| = m, |\varepsilon| = 1}{\| \mathbf{1}_{\varepsilon A}\|} \quad \mbox{, } \quad   h_{r}^*(m) = \sup_{|A| = m, |\varepsilon| = 1}{\| \mathbf{1}_{\varepsilon A}\|}\, \quad\quad  (m \in \N)\,. \]
	
	\begin{defn}
	We say that $\mathcal B$ is \emph{democratic} (resp. \emph{superdemocratic}) if there exists $C > 0$ such that $h_{r}(m) \leq C \, h_{l}(m)$ (\, resp. $h_{r}^*(m) \leq C \, h_{l}^*(m)$ \,) for every $m \in \N$.
	\end{defn}
	
	Another characterization of greedy bases was more recently provided by \'O. Blasco and the first author by means of the \emph{best $m$-th error in the approximation using polynomials of constant $($resp. modulus-constant$)$ coefficients}: 
\begin{eqnarray*}
\mathcal D_m(x,\mathcal B)_{\mathbb X} & = & \mathcal D_m(x)=\inf\lbrace\Vert x-\alpha \mathbf{1}_{A}\Vert : \alpha\in\mathbb R, \, A\subset\mathbb N, \, \vert A\vert=m\rbrace\\[1mm]
\mathcal{D}_m^*(x,\mathcal B)_{\mathbb X} & = & \mathcal{D}_m^*(x)=\inf\lbrace\Vert x-\alpha \mathbf{1}_{\varepsilon A}\Vert : \alpha\in\mathbb R, \,  A\subset\mathbb N, \, \vert A\vert=m, \vert\varepsilon\vert=1\rbrace
\end{eqnarray*}	
	
	\begin{thm}{\cite[Corollary 1.8]{BB}}\label{bb}
Let $\mathcal{B}$ be a basis of a Banach space $\mathbb{X}$. The following assertions are equivalent:
	\begin{enumerate}\itemsep0.3em
	\item[(i)] $\mathcal{B}$ is greedy;
	\item[(ii)] There is $C>0$ such that $\Vert x-\mathcal G_m(x)\Vert\leq C\,  \mathcal{D}_m(x)$ for every $x \in \mathbb{X}$ and $m \in \N$. 
    \item[(iii)] There is $C>0$ such that $\Vert x-\mathcal G_m(x)\Vert\leq C \, \mathcal{D}_m^\ast(x)$ for every $x \in \mathbb{X}$ and $m \in \N$. 
	\end{enumerate}
	\end{thm}
	
	
The striking feature of this theorem compared to \eqref{equa:defnGreedyAux1} is that, while $\lim_{m}\sigma_{m}(x) = 0$ for every $x \in \mathbb{X}$, the terms $\mathcal{D}_{m}^{\ast}(x)$ and $\mathcal{D}_{m}(x)$ do not necessarily converge to zero if $x \neq 0$. Indeed, we have the following examples:

\begin{itemize}\itemsep0.3em
\item[$\rhd$] {\cite[Theorem 3.2]{BB}},{\cite[Theorem 1.4]{BB2}} If $\mathbb{X} = \mathbb{H}$ is a (separable) Hilbert space and $\mathcal{B}$ is an orthonormal basis, then
		\begin{equation}\label{equa:precedent1}
		\lim_{m\rightarrow \infty}\mathcal D_m(x) = \lim_{m\rightarrow \infty}\mathcal D_m^*(x)  = \Vert x\Vert\,, \quad \mbox{ for every } x\in\mathbb H.
		\end{equation}
\item[$\rhd$] \cite[Proposition 3.4]{BB} If $\mathbb{X} = \ell^{p}$ ($1 < p <\infty$) and $\mathcal{B}$ is the canonical basis, then 
		\begin{equation} \label{equa:precedent2}
		\lim_{m \rightarrow +\infty}{\mathcal{D}_{m}(\mathbf{1}_{B})} = \lim_{m \rightarrow +\infty}{\mathcal{D}_{m}^*(\mathbf{1}_{B})} = \| \mathbf{1}_{B}\|  \,, \quad \mbox{ for every finite }B \subset \N\,.
		\end{equation}
\end{itemize}
	
In the present paper, we aim to delve into this aspect.
Let us briefly explain the structure of the paper. In Section \ref{section2} we show that $\mathcal{D}_{m}^{\ast}(x)$ and $\mathcal{D}_{m}(x)$ do not converge to zero as $m \rightarrow + \infty$ for any $x \neq 0$. In Section 	\ref{section3} we prove the main result  of the paper (Theorem \ref{thm:main}), namely a characterization of those bases $\mathcal{B}$ for which there is a positive constant $c>0$ such that
\[ c \| x\| \leq \liminf_{m \rightarrow + \infty}{\mathcal{D}_{m}^{\ast}(x)} \leq \limsup_{m \rightarrow + \infty}{\mathcal{D}_{m}^{\ast}(x)} \leq \| x\| \quad \quad \mbox{ for every }x \in \mathbb{X}\,, \]
in terms of the democracy and superdemocracy functions.
We also provide a quite general condition ensuring that
\[ \lim_{m \rightarrow + \infty}{\mathcal{D}^*_{m}(x)} = \| x\| \quad \mbox{ for every } x \in \mathbb{X}\,. \]
In Section \ref{section4} we deal with the notion of almost-greedy bases. We study how this property can be also characterized in terms of polynomials of constant or modulus-constant coefficients, extending a recent result of S. J. Dilworth and D. Khurana in \cite{DK1}.\\ 

Let us point out \cite{AK} as our basic reference for notation and fundamental results on greedy basis.
	\end{section}

\section{The limit of errors $\mathcal D_m^*(x)$ and $\mathcal D_m(x)$ is nonzero}\label{section2}
Since $\mathcal D_m^*(x)\leq \mathcal D_m(x) \leq \| x\|$ for every $m\in\mathbb N$ and every $x\in\mathbb X$, it is only necessary to study lower bounds of $\mathcal D_m^*(x)$.

\begin{prop}\label{prop:nonzeroLimitError}
Let $\mathcal{B} = (e_{n})_{n=1}^{\infty}$ be a basis of a Banach space $\mathbb{X}$. Then, for every $x \in \mathbb{X}$ 
\[ \frac{1}{4 K_{b}} \, \sup_{n \in \N}| e_{n}^{\ast}(x)| \, \leq \,  \liminf_{m \rightarrow \infty} \mathcal{D}_{m}^{\ast}(x)\,. \]
\end{prop}

\begin{proof}
Let $x \in \mathbb{X}$. Note that for every finite set $A \subset \N$, $\alpha \in \R$ and $|\varepsilon| = 1$ it holds that
\[ \|  x - \alpha \mathbf{1}_{\varepsilon A}\|  \geq \sup_{n \in \N} \frac{|e_{n}^{\ast}( x - \alpha \mathbf{1}_{\eta A})|}{\| e_{n}^{\ast}\|} \geq \frac{\sup_{n \in \N}{|e_{n}^{\ast}( x - \alpha \mathbf{1}_{\varepsilon A})|}}{2 K_{b}} \,  \geq \frac{\sup_{n \in \N} \big| |e_{n}^{\ast}(x)| - |\alpha| \big|}{2K_{b}} \,  \,. \]
Let us also fix $\delta > 0$ and $n_{0} \in \N$ with the property that
\[ |e_{n}^{\ast}(x)| \leq \delta \quad \mbox{ for every } n \geq n_{0}\,. \] 
If $A$ satisfies $|A| > n_{0}$, then there is $j \in A$ with $j > n_{0}$, and so 
\[ \| x - \alpha \mathbf{1}_{\varepsilon A}\| \geq  \frac{| e_{j}^{\ast}(x) - |\alpha| |}{2 K_{b}} \geq \frac{| |\alpha| - \delta |}{2 K_{b}} \,.  \]
In particular, combining both lower estimations we get that for $|A| > n_{0}$
\begin{eqnarray*}
\| x - \alpha \mathbf{1}_{\varepsilon A}\| \geq \frac{| |\alpha| - \delta | + \sup_{n \in \N} \big| |e_{n}^{\ast}(x)| - |\alpha| \big|}{4 K_{b}} \geq \sup_{n \in \N}{\frac{|e_{n}^{\ast}(x)| - \delta}{ 4 K_{b}}}\,.
\end{eqnarray*}
Therefore, for $m > n_{0}$
\[  \mathcal{D}_{m}^{\ast}(x) \geq \sup_{n \in \N}{\frac{|e_{n}^{\ast}(x)| - \delta}{ 4 K_{b}}}\,. \]
\end{proof}

\section{Main result: equivalence with the norm} \label{section3}

The issue of when $\liminf_{m}{\mathcal{D}^{\ast}_{m}(x)}$ (resp. $\liminf_{m}{\mathcal{D}_{m}(x)}$) is equivalent to $\| x\|$ is going to be determined by the behaviour of the superdemocracy functions (resp. democracy functions),  see Section \ref{section1} for the definitions. Along the present section we are going to focus on proving the results for superdemocracy case, namely for $h_{l}^\ast(m)$,  $h_{r}^\ast(m)$ and the error $\mathcal{D}^{\ast}_{m}(x)$. The arguments for the case $h_{l}(m)$,  $h_{r}$ and the error $\mathcal{D}_{m}(x)$ are completely analogous. First of all, we recall a trivial estimates of the superdemocray functions for any basis:
\[  h_{l}^{\ast}(k) \leq K_{b} \, h_{l}^*(m)\,, \quad h_{r}^{\ast}(k) \leq K_{b} \, h_{r}^{\ast}(m)\quad  \mbox{ for every } k \leq m\,. \]
These relations together with the trivial inequality $h_{l}^{\ast}(m) \leq h_{r}^{\ast}(m)$ $(m \in \N)$ yield that there are three possible cases:
\begin{itemize}\itemsep0.2em
\item[$\rhd$] $h_{l}^{\ast}(m)$ and $h_{r}^{\ast}(m)$ are bounded.
\item[$\rhd$] $h_{l}^\ast(m)$ is bounded and $h_{r}^{\ast}(m) \rightarrow + \infty$ as $m \rightarrow + \infty$.
\item[$\rhd$] $h_{l}^\ast(m), h_{r}^{\ast}(m) \rightarrow + \infty$ as $m \rightarrow + \infty$.
\end{itemize}

\begin{defn}
The functions $h_{l}^\ast(m)$ and $h_{r}^\ast(m)$ (resp. $h_{l}(m)$ and $h_{r}(m)$) are said to be \emph{comparable} if they are both bounded or divergent to infinity.
\end{defn}

The main result of the section is the following theorem.

\begin{thm}\label{thm:main}
Let $\mathcal{B}$ be a basis of a Banach space $\mathbb{X}$. The following assertions are equivalent:
\begin{enumerate}
\item[(i)] There is a positive constant $c> 0$ such that
\[ c \, \| x\|  \leq \liminf_{m \rightarrow + \infty}{\mathcal{D}^*_{m}(x)} \leq \limsup_{m \rightarrow +\infty}{\mathcal{D}^*_{m}(x)} \leq \| x\| \quad \mbox{ for every } x \in \mathbb{X}. \]
\item[(ii)] $h_{l}^*(m)$ and $h_{r}^*(m)$ are comparable. 
\end{enumerate}
Moreover, if $\mathcal{B}$ is monotone and $h_{l}^*(m) \rightarrow +\infty$ as $m \rightarrow + \infty$, then
\begin{equation}\label{equathmMainAux1} 
\lim_{m \rightarrow +\infty}{\mathcal{D}^*_{m}(x)} = \| x\|\,.\\[1mm] 
\end{equation}
$($The theorem also holds if we replace $\mathcal{D}_{m}^\ast(x)$, $h_{l}^*(m)$, $h_{r}^*(m)$ respectively by $\mathcal{D}_{m}(x)$, $h_{l}(m)$, $h_{r}(m)$.$)$ 
\end{thm}

\noindent Before going into the proof let us make a few observations:

\begin{itemize}\itemsep1em
\item[$\rhd$] From Theorem \ref{thm:main} we recover  \eqref{equa:precedent1} and \eqref{equa:precedent2}. Indeed, if $\mathbb{H}$ is a (separable) Hilbert space and $\mathcal{B}$ is an orthonormal basis of $\mathbb{H}$ then $h_{l}(m) = h_{l}^*(m) = m^{1/2}$. On the other hand, for $\mathbb{X} = \ell_{p}$ with $1 \leq p < \infty$ and $\mathcal{B}$ is the canonical basis, it holds that $h_{l}(m) = h_{l}^*(m) = m^{1/p}$.
\item[$\rhd$] For $\mathbb{X} = L_{p}[0,1]$ we have that the  Haar basis $\mathcal{B}$ is monotone (see \cite[Theorem 5.18]{H}) and satisfies $h_l^*(m)=h_l(m)\approx m^{1-1/p}$ for $1\leq p<\infty$. Hence, it satisfies that $\lim_{m}{\mathcal{D}^*_{m}(x)} = \lim_{m}{\mathcal{D}_{m}(x)} = \| x\|$ for every $x \in \mathbb{X}$.

\item[$\rhd$] If $\mathcal{B}$ is superdemocratic (resp. democratic), then it satisfies Theorem \ref{thm:main}.(ii)  (resp. Theorem \ref{thm:main}.(ii) for $h_{r}(m)$ and $h_{l}(m)$). However, there are easy examples showing that converse is not true. For instance, the canonical basis of  $\ell^2\oplus_{1} \ell^4$ satisfies that $h_l(m)=h_l^*(m)\approx m^{1/4}$ and $h_r(m)=h_r^*(m)\approx m^{1/2}$.

\item[$\rhd$] \emph{Example of basis not satisfying Theorem \ref{thm:main}.(ii)}: Let us consider $\mathbb{X} = \ell_{1}$ and let $\mathcal{B} = (\mathbf{x}_{n})_{n=1}^{\infty}$ be the \emph{difference basis}, which in terms of the canonical basis $(e_{n})_{n = 1}^{\infty}$ is given by
	\[ \mathbf{x}_1 = e_1\,, \quad  \mathbf{x}_n = e_n - e_{n-1}\,, \quad n=2,3,... \]
	By \cite[Lemma 8.1]{BBGHO}, it holds that $h_l^*(m)=h_l(m)=1$ and $h_r^*(m)=h_r(m)=2m$.
	
\item[$\rhd$] \emph{Example of basis satisfying $\lim_{m}{\mathcal{D}_{m}(x)} = \| x\|$ for every $x \in \mathbb{X}$, but $\liminf_{m}{\mathcal{D}^{\ast}(x)}$ is not even equivalent to $\| x\|$}: Let $\mathbb{X} = \mathbf{c}$ be the space of convergent sequences and let $\mathcal{B} = ( \mathbf{s}_{n})_{n=1}^{\infty}$ be the \emph{summing basis}, defined as
\[ \mathbf{s}_{n}:=(\underbrace{0, \ldots, 0}_{n-1}, 1,1, \ldots)\,, \quad n \in \N \,.\]
By \cite[Lemma 8.1]{BBGHO} we know that $h_l^*(m) \approx 1$ and $h_r^* (m)\approx m$, so Theorem \ref{thm:main}.(ii) does not hold. On the other hand, $\mathcal{B}$ is monotone and $h_l(m)\approx h_r(m)\approx m$ by the same reference. Thus, $\lim_{m}{\mathcal{D}_{m}(x)} = \| x\|$ for every $x \in \mathbb{X}$.
	\item[$\rhd$] \emph{Condition Theorem \ref{thm:main}.(ii) is not preserved for dual bases}: If $(e_{n})_{n=1}^{\infty}$ is the canonical basis of $\ell_{1}$, let us consider the sequence  $\mathbf{x}_n = e_n - (e_{2n+1}+e_{2n+2})/2$, $n\in \N$ and the space
	\[ \mathbb X:=\overline{\operatorname{span}\lbrace \mathbf{x}_n : n\in\mathbb N\rbrace}^{\ell^1}\,. \]
This is known as the \emph{Lindenstrauss space} \cite{L} and the sequence $\mathcal{B} = (\mathbf{x}_{n})_{n=1}^{\infty}$ is actually a monotone basis for $\mathbb{X}$ (see \cite[pg 457]{S}).  In \cite[Section 8.2]{BBGHO} it is shown that $h_l^* (m)\approx m$. On the other hand, in the same reference it is proved that the dual space $\mathbb{X}^{\ast}$ with the corresponding dual basis $\mathcal{B}^{\ast}$ satisfies  $h_l^* (m)\approx 1$ and $h_r^*(m)\approx \ln(m)$.
\end{itemize}

\subsection{Proof of the main result}

\begin{prop}\label{prop:main}
Let $\mathcal{B}$ be a basis of a Banach space $\mathbb{X}$. Then,
	\begin{eqnarray}
\label{equa:p1aux1}	 \sup_{\substack{A \subset \N\\ finite},\vert\eta\vert=1}{ \, \liminf_{m \rightarrow + \infty} \mathcal D_{m}^*(\mathbf{1}_{\eta A})} \, & \leq & \,  (1 + K_b) \,  \liminf_{m \rightarrow + \infty}{h_{l}^*(m)}\,\,\, \leq \,\,\, \infty\,,\\
\label{equa:p1aux2}	 \sup_{\substack{A \subset \N\\ finite}}{ \, \liminf_{m \rightarrow + \infty} \mathcal D_{m}(\mathbf{1}_{A})} \, & \leq & \,  (1 + K_b) \,  \liminf_{m \rightarrow + \infty}{h_{l}(m)}\,\,\, \leq \,\,\, \infty\,.  
	\end{eqnarray}
\end{prop}

\begin{proof}
We explain the argument for \eqref{equa:p1aux1}, as the proof of \eqref{equa:p1aux2} is completely analogous with the obvious replacements. Let us fix a finite set $A \subset \mathbb{N}$ and $\eta \in \{ \pm 1\}^{A}$, and let us take $\lambda \in \mathbb{R}$ satisfying
	\begin{equation}\label{equa:p1aux3} 
	\lambda < \liminf_{m \rightarrow  + \infty}{\mathcal D_{m}^*(\mathbf{1}_{\eta A})}. 
	\end{equation}
We can then find  $m_{0}, n_{0} \in \N$ with the following properties:
	\begin{itemize}\itemsep0.3em 
		\item[$\rhd$] $ \lambda \leq \| \mathbf{1}_{\eta A} - \alpha \mathbf{1}_{\varepsilon B}\| $  for every $\alpha \in \mathbb{R}$, $\vert\varepsilon\vert=1$ and $B \subset \N$ with $|B| \geq m_{0}$\,,
		\item[$\rhd$] $A \subset \{ 1, \ldots, n_{0}\}$\,.
	\end{itemize}
	Let $C \subset \N$ be a finite set with $|C| \geq m_{0} + n_{0}$. Then,
	\[ \mathbf{1}_{\varepsilon C} - P_{n_{0}} (\mathbf{1}_{\varepsilon C}) =  \mathbf{1}_{\varepsilon C'}\,  \]
	where  $C' := C \setminus \{ 1, \ldots, n_{0}\}$. Notice that $|C'| \geq m_{0}$, so in particular 
	\[ \lambda \leq \| \mathbf{1}_{\eta A} - \mathbf{1}_{(\eta A) \cup (\varepsilon C')} \| = \| \mathbf{1}_{\varepsilon C'}\| \leq \| \Id - P_{n_{0}}\| \, \| \mathbf{1}_{\varepsilon C}\| \leq (1 + K_{b}) \, \| \mathbf{1}_{\varepsilon C}\|. \]
	Thus, we have the relation
	\[ \lambda \leq (1 + K_b) \, \liminf_{m \rightarrow + \infty}{h_{l}^ \varepsilon(m)}. \]
Taking supremums on $\lambda$ according to \eqref{equa:p1aux3} we conclude that 
	\[ \liminf_{m \rightarrow  + \infty}{\mathcal D_{m}^*(\mathbf{1}_{\eta A})} \,\, \leq \,\, (1 + K_b) \, \liminf_{m \rightarrow + \infty}{h_{l}^ \varepsilon(m)}. \]
\end{proof}

\begin{thm}\label{thm:sameBehaviourImpliesEquivalence}
Let $\mathcal{B}$ be a basis of a Banach space $\mathbb{X}$. Assume that there is a constant $C>0$ satisfying
\[ \sup_{n \in \N} h_{r}^*(n) \leq C \, \sup_{n \in \N} h_{l}^*(n) \, \leq \infty \,. \]
Then, for every $x \in \mathbb{X}$
\begin{equation}\label{equa:sameBehaviourImpliesEquivalenceAux1} 
\frac{1}{C + K_{b}(1 + C)} \, \| x\| \, \leq \, \liminf_{m}{\mathcal{D}^*_{m}(x)} \, \leq \, \limsup_{m}{\mathcal{D}_{m}^*(x)} \, \leq \, \| x\|\,. 
\end{equation}
\end{thm}

\begin{proof}
Let us fix $x \in \mathbb{X}$. We just have to show that the left hand-side of \eqref{equa:sameBehaviourImpliesEquivalenceAux1}  holds. For, let $0<\delta <1$ and $m_{0}, n_{0} \in \mathbb{N}$ such that
\begin{eqnarray*}
\| P_{n}(x) - x\| &\leq& \delta \, \| x\| \quad \mbox{ for every } n \geq n_{0 }\,, \\[1mm]
h_{r}^*(n_{0}) &\leq&  \, C \, (1-\delta) \, h_{l}^*(m_{0}) \,.
\end{eqnarray*} 
Given $\alpha \in \R$, $A \subset \N$ with $|A| \geq m_{0} + n_{0}$ and $\varepsilon \in \{ \pm 1\}^{A}$, we are going to establish two lower bounds for $\| x - \alpha \mathbf{1}_{\varepsilon A}\|$. 
\begin{itemize}
\item[$\rhd$] Since $|A \cap (n_{0}, +\infty)| \geq m_{0}$ we can find $n \geq n_{0}$ such that $|A \cap (n, + \infty)| = m_{0}$. Thus, applying the operator $\Id - P_{n}$  to $x - \alpha \mathbf{1}_{\varepsilon A}$ we have that
\begin{equation}\label{equa:auxMainTheoEstimation1}
\| x- \alpha \mathbf{1}_{\varepsilon A} \| \geq \frac{1}{K_{b} + 1} \| (\Id - P_{n})(x) - \alpha \mathbf{1}_{\varepsilon (A \cap (n,+\infty))} \| \geq \frac{1}{K_{b} + 1} \big( |\alpha| \, h_{l}^*(m_{0}) - \delta \, \| x\| \big)\,.
\end{equation}
\item[$\rhd$] As $|A| \geq n_{0}$ we can find $n \geq n_{0}$ with $|A \cap [1, n]| = n_{0}$, so that
\begin{eqnarray}
\| x - \alpha \mathbf{1}_{\varepsilon A}\| \geq \frac{1}{K_{b}} \, \big( \| P_{n}(x) - \alpha \mathbf{1}_{\varepsilon (A \cap [1,n])} \| \big) &\geq& \frac{1}{K_{b}} \, \big( \| x\|(1 - \delta) - |\alpha| \, h_{r}^*(n_{0})  \big)\\
\label{equa:auxMainTheoEstimation2} & \geq& \frac{1- \delta}{K_{b}} \, \big( \| x\| - C \, |\alpha| \, h_{l}^*(m_{0}) \big)
\end{eqnarray}
\end{itemize}
Note that the lower estimations  \eqref{equa:auxMainTheoEstimation1} and \eqref{equa:auxMainTheoEstimation2} are respectively increasing and decreasing linear functions $f(t)$ and $g(t)$ on $t=|\alpha|$. Moreover these functions have a unique point of intersection $t_{0} > 0$ which can be easily checked to satisfy
\begin{equation} 
t_{0} = \frac{\| x\|}{h_{l}^*(m_{0})} \cdot \frac{ (1- \delta) \, (1+K_{b}) + \delta \, K_{b}}{C(1- \delta) (1 + K_{b}) + K_{b} }\,.
\end{equation}
Thus
\[ \| x - \alpha \mathbf{1}_{\varepsilon A}\| \geq \max{\{f(|\alpha|), g(|\alpha|)\}} \geq f(t_{0}) = g(t_{0}) = \frac{\| x\|}{1 + K_{b}} \left[ \frac{ (1- \delta) \, (1+K_{b}) + \delta \, K_{b}}{C(1- \delta) (1 + K_{b}) + K_{b} } - \delta  \right] \,. \]
Taking the infimum of $\| x - \alpha \mathbf{1}_{\varepsilon A}\|$ on $\alpha \in \mathbb{R}$ and $A$ satisfying the conditions above, we deduce that
\[ \liminf_{k \rightarrow + \infty}{\mathcal{D}_{k}^*(x)} \geq  \inf_{k \geq m_{0} + n_{0}}{\mathcal{D}_{k}^*(x)} \geq \frac{\| x\|}{1 + K_{b}} \left[ \frac{ (1- \delta) \, (1+K_{b}) + \delta \, K_{b}}{C(1- \delta) (1 + K_{b}) + K_{b} } - \delta  \right]\,. \]
Finally, making $\delta \rightarrow 0^{+}$ we get the desired conclusion.
\end{proof}

\begin{proof}[Proof of Theorem \ref{thm:main}]
To check (i) $\Rightarrow$ (ii),  note that using Proposition \ref{prop:main} we  then deduce that
\[ \sup_{m \in \N}{h_{r}^*(m)} =  \sup_{\substack{A \subset \N\\ finite},\vert\eta\vert=1}{ \| \mathbf{1}_{\eta A}\|} \, \leq  \,\sup_{\substack{A \subset \N\\ finite},\vert\eta\vert=1}{ \, \liminf_{m \rightarrow + \infty} \mathcal{D}_{m}^*(\mathbf{1}_{\eta A})} \, \leq  \,  (1 + K_b) \,  \liminf_{m \rightarrow + \infty}{h_{l}^*(m)} \, \leq \,  \infty. \]
It is clear from this inequality that $h_{l}^{\ast}(m)$ and $h_{r}^{\ast}(m)$ are then comparable. To see the converse  (ii) $\Rightarrow$ (i), note first that if $h_{l}^{\ast}(m)$ and $h_{r}^{\ast}(m)$ are comparable, then there exists $C>0$ such that \begin{equation}\label{equa:thmMainAux}
\sup_{m \in \N}{ h_{r}^{\ast}(m)} \leq \sup_{m \in \N}{ C \, h_{l}^{\ast}(m)}
\end{equation}
and so Theorem \ref{thm:sameBehaviourImpliesEquivalence} applies. The second statement of the theorem follows also from Theorem \ref{thm:sameBehaviourImpliesEquivalence} since $\mathcal{B}$ being monotone means that $K_{b} = 1$, and condition $\lim_{m} h_{l}^{\ast}(m) = + \infty$ means that \eqref{equa:thmMainAux} holds for every $C>0$.
\end{proof}

\section{Almost-greediness and polynomials with constant coefficients}\label{section4}

\begin{defn}
Let $\mathcal{B} = (e_{n})_{n=1}^{\infty}$ be a basis of a Banach space $\mathbb{X}$. We say that $\mathcal{B}$ is \emph{almost-greedy} if there exists a constant $C>0$ such that 
\[ \Vert x-\mathcal G_m(x)\Vert \leq C \, \widetilde{\sigma}_m(x) \]
where 
\[ \widetilde{\sigma}_m(x,\mathcal B)_{\mathbb X}=\widetilde{\sigma}_m(x):=\inf\lbrace \Vert x-\sum_{n \in A}{e_{n}^{\ast}(x) \, e_{n}}\Vert : A\subset\mathbb N, \vert A\vert=m\rbrace.\]
\end{defn}
 This notion was introduced  by S. J. Dilworth, N. J. Kalton, D. Kutzarova and V. N. Temlyakov in \cite{DKKT}, together with two characterizations. First, that a basis is almost-greedy if and only if it is quasi-greedy and democratic. The second characterization  is given in the next theorem.
 
\begin{thm}[{\cite[Theorem 3.3]{DKKT}}]\label{dkkt2}
Let $\mathcal{B}$ be a basis of a Banach space $\mathbb{X}$. Then, $\mathcal{B}$ is almost-greedy if and only if for some (resp. every) $\lambda>1$, there exists a positive constant $C_\lambda$ such that 
$$\Vert x-\mathcal G_{[\lambda m]}(x)\Vert \leq C_\lambda \sigma_m(x)\,, \quad  \mbox{ for every } x\in\mathbb X,\, m\in\mathbb N.$$
Indeed, $C_\lambda \approx \frac{1}{\lambda-1}$.
\end{thm}

As in the case of greedy basis, we can replace the error $\sigma_{m}(x)$ by the $m$-th error of approximation by polynomials with constant (resp. modulus-constant) coefficients.

\begin{thm}
Let $\mathcal{B}$ be a basis of a Banach space $\mathbb{X}$ and let $\lambda > 1$. The following assertions are equivalent:
\begin{enumerate}\itemsep0.5em
\item[(i)] $\mathcal{B}$ is almost-greedy.
\item[(ii)] There is $C>0$ such that $\Vert x-\mathcal G_{[\lambda m]}(x)\Vert \leq C_\lambda \, \mathcal{D}_{m}(x)$  for every $x\in\mathbb X$ and every $m\in\mathbb N$.
\item[(iii)] There is $C>0$ such that $\Vert x-\mathcal G_{[\lambda m]}(x)\Vert \leq C_\lambda \, \mathcal{D}_{m}^*(x)$  for every $x\in\mathbb X$ and every $m\in\mathbb N$.
\end{enumerate}
\end{thm}

\begin{proof}
Implication (i) $\Rightarrow$ (iii) $\Rightarrow$ (ii) are clear using Theorem \ref{dkkt2} and the inequalities $\sigma_{m}(x) \leq \mathcal D_m^*(x) \leq \mathcal D_m(x)$. To show that (ii) $\Rightarrow$ (i) we follow the ideas from the proof of Theorem \ref{dkkt2}: using the hypothesis, we argue that $\mathcal{B}$ is democratic and quasi-greedy. 

To see that it is democratic, let $m \in \N$ and $A, B \subset \N$ with $|A| = m$ and $|B| = [\lambda m]$. Let us consider a set $E \supset A, B$ with $|E| = m + [\lambda m]$, let $\delta > 0$ and consider the element $x = \mathbf{1}_{A} + (1 + \delta) \mathbf{1}_{E \setminus A}$. Then,
\[ \| \mathbf{1}_{A}\| = \| x - \mathcal{G}_{[\lambda m]}(x)\| \leq C_{\lambda} \, \mathcal{D}_{m}(x) \leq C_{\lambda} \| \mathbf{1}_{B \setminus A} + (1 + \delta) \mathbf{1}_{B \cap A}\|\,. \]
As $\delta > 0$ is arbitrary, taking supremum over $A$ and infimum over $B$ we deduce that
\[ h_{r}(m) \leq C_{\lambda} \, h_{l}(\lambda m) \leq C_{\lambda} \, K_{b} \, h_{l}(m)\,, \]
where in the last inequality we have used the estimations mentioned at the beginning of Section \ref{section2}. 

Let show now that the basis $\mathcal{B}$ is  quasi-greedy. For, take $m\in\mathbb N$ and $r\in\mathbb N\cup\lbrace 0\rbrace$ such that $[\lambda r]\leq m<[\lambda(r+1)]$. Then,  
\[ \Vert x-\mathcal G_m(x)\Vert \leq \Vert x-\mathcal{G}_{[\lambda r]}(x)\Vert + \Vert \mathcal{G}_{[\lambda r]}(x)-\mathcal G_m(x)\Vert\,. \]
Note that $\mathcal{G}_{[\lambda r]}(x)-\mathcal G_m(x)$ contains at most $m-[\lambda r]<\lambda$ summands of the form $e_{n}^{\ast}(x) \, e_{n}$, so that
\[ \Vert \mathcal{G}_{[\lambda r]}(x)-\mathcal G_m(x)\Vert  \leq \big( \lambda \, \sup_{n \in \N}{\| e_{n}\|} \, \sup_{n \in \N}{\| e_{n}^{\ast}\|} \big) \,\| x\|  \,.  \]
On the other hand, using the hypothesis
\[ \| x - \mathcal{G}_{[\lambda r]}(x)\| \leq C_{\lambda} \, \mathcal{D}_{m}(x) \leq C_{\lambda} \, \| x\|\,. \]
Thus, the basis is quasi-greedy.
\end{proof}

Recently, S. J. Dilworth and D. Khurana  provided the following characterization of almost-greedy bases in the same spirit of Theorem \ref{bb}. In order to present it we have to introduce some notation: if $A, B \subset \N$ are finite sets, we will write $A < B$ if $\max{A} < \min{B}$. 
\begin{eqnarray*} 
\mathcal H_m(x) & := & \inf\lbrace \Vert x-\alpha\mathbf{1}_A\Vert \colon \alpha\in\mathbb R\,, \vert A\vert =m\; \text{and either}\; A<\Lambda_{m}(x)\; \text{or}\; A>\Lambda_{m}(x)\rbrace
\end{eqnarray*}
where recall that $\Lambda_{m}(x)$ is the $m$-th greedy set associated to $x$ introduced in Section \ref{section1}.

\begin{thm}{\cite{DK1}}\label{thm:DK1}
Let $\mathcal{B}$ be a basis of a Banach space $\mathbb{X}$. Then, $\mathcal{B}$ is almost-greedy if and only if there exists $C > 0$ such that  
\[ \| x - \mathcal{G}_{m}(x)\| \, \leq \, C \, \inf_{1 \leq n \leq m}{\mathcal{H}_{n}(x)}\quad \mbox{for every $x \in \mathbb{X}$} \mbox{ and  } m \in \N. \]
\end{thm}

Inspiring on the previous theorem , we can prove the following result which is again strinking as $\mathcal{D}_{m}(x) \leq \mathcal{H}_{m}(x)$ and so $\liminf{\mathcal{H}_{m}(x)} \approx \| x\|$ when $h_{l}(m)$ and $h_{r}(m)$ are comparable by Theorem \ref{thm:main}.

\begin{cor}\label{coro:lastOne}
Let $\mathcal{B}$ be a basis of a Banach space $\mathbb{X}$. Then, $\mathcal{B}$ is almost-greedy if and only if there exists $C > 0$ such that  
\begin{equation}\label{equa:coroLastOneAux1} 
\| x - \mathcal{G}_{m}(x)\| \leq C \, \mathcal{H}_{m}(x) \quad  \mbox{ for every } x \in \mathbb{X} \mbox{ and  } m \in \N.
\end{equation}
\end{cor}

\begin{proof}
If $\mathcal{B}$ is quasi-greedy then \ref{equa:coroLastOneAux1}  holds by Theorem \ref{thm:DK1}. To see the converse we use the aforementioned characterization of almost-greedy bases as those being quasi-greedy and democratic. The fact that $\mathcal{B}$ is quasi-greedy follows from the hypothesis and the trivial inequality $\mathcal{H}_{m}(x) \leq \| x\|$.	Let us show that $\mathcal{B}$ is democratic. Let $A,B \subset \N$ be finite subsets of cardinality $m$, and take $E \subset \N$ also with $|E| = m$ and moreover $A < E$ and $B < E$. Fixed $\delta > 0$ consider the elements $x=\mathbf{1}_{A} + (1 + \delta) \mathbf{1}_{E}$ and $y =  \mathbf{1}_{E} + (1+ \delta) \mathbf{1}_{B}$. Then,
\[ \|\mathbf{1}_{A}\| = \| x - \mathbf{1}_{E}\| = \| x - \mathcal{G}_{m}(x)\| \leq C \, \mathcal{H}_{m}(x) \leq C \, \|x -  \mathbf{1}_{A}\| = C \,(1+ \delta) \, \| \mathbf{1}_{E}\|\,. \]
Analogously,
\[ \|\mathbf{1}_{E}\| =  \| y - \mathbf{1}_{B}\| = \| y - \mathcal{G}_{m}(y)\| \leq C \, \mathcal{H}_{m}(y) \leq C \|y -  \mathbf{1}_{E}\| = C \, (1+ \delta) \, \| \mathbf{1}_{B}\|\,. \]
Since $\delta>0$ was arbitraty, we conclude that  $h_{r}(m) \leq C^{2} \, h_{l}(m)$ for every $m \in \N$, and so the basis is democratic.
\end{proof}

\end{document}